\documentclass[10pt,a4paper]{article}
\usepackage{enumitem}
\newlist{steps}{enumerate}{1}
\setlist[steps, 1]{label = Step \arabic*:}
\usepackage{color}
\usepackage{commath}
\usepackage{graphicx}
\usepackage{adjustbox}
\usepackage{multirow}  
\usepackage{rotating}
\usepackage{csquotes}
\usepackage{caption}
\usepackage{subcaption}
\usepackage{enumitem}
\usepackage{float}
\usepackage{todonotes}
\usepackage{amsthm,amssymb,mathrsfs,setspace,amsmath}
\numberwithin{figure}{section}
\usepackage[bindingoffset=0.2in, left=0.5in, right=0.5in, top=1in, bottom=1in, footskip=.25in]{geometry}
\usepackage[colorlinks,citecolor=blue,urlcolor=blue,bookmarks=false,hypertexnames=true]{hyperref}
\hypersetup{citecolor=blue} 
\usepackage{color}
\usepackage[sort,numbers]{natbib}
\newtheorem{definition}{Definition}[section]

\newtheorem{theorem}{Theorem}[section]

\newtheorem{example}{Example}
\newcommand{\x}{F}

\newcommand{\z}{t}
\newcommand{\f}{\Gamma}
\newcommand{\e}{T}
\newcommand{\w}{\mathcal{W}}
\newcommand{\2}{k}
\newcommand{\3}{n}
\newcommand{\Z}{z}
\newcommand{\g}{m}
\newcommand{\h}{\rho}
\newcommand{\1}{s}
\newcommand{\4}{g}
\newcommand{\5}{M}
\newcommand{\6}{\Delta}
\newcommand{\R}{r}
\newcommand{\8}{l}
\newcommand{\7}{\zeta}
\newcommand{\9}{\Lambda}
\newcommand{\n}{V}

\title{ Semi-orthogonal Tribonacci Wavelets and Numerical Solutions of Nonlinear Singular BVPs Arising in a Chemical Reaction} 
\author{Ankita Yadav$^a$\thanks{$^a$ankita\_2321ma01@iitp.ac.in}, Amit K. Verma$^b$\thanks{$^b$Corresponding author: akverma@iitp.ac.in}, \\\small{\it{$^{a,b}$ Department of Mathematics,}} \\\small{\it{Indian Institute of Technology Patna,}}\\\small{\it{ Bihta, Patna 801103, (BR) India.}}} 
\date{\today}

\begin{document}

\maketitle

\begin{abstract}
    In this article, we introduce a  semi-orthogonal tribonacci wavelet and develop a  semi-orthogonal tribonacci wavelet collocation method, offering an effective numerical method for solving a class of  non-linear singular BVPs.  
\end{abstract}
\section{Introduction} 
In thermal explosion theory, the critical condition for ignition is defined as the point where the heat generated by the chemical reaction exactly balances the heat lost to the surroundings. The governing equation for thermal balance, accounting for the heat generated by the chemical reaction and the heat conducted away, can be expressed as \cite{chambre1952solution}
\begin{equation}\label{P3_E1}
     \9 \nabla^{2} \mathcal{T} = - Q \n,
\end{equation}
where $\mathcal{T}$ is the gas temperature, Q is the  heat of reaction, $\9$ is the thermal  conductivity, $\n$ is the  reaction velocity, and $  \nabla^{2} $ is the  Laplacian operator. The reaction \eqref{P3_E1} is assumed to be monomolecular, with its rate governed by the Arrhenius law
\begin{equation}\label{P3_E2}
    \n = c\,a \exp \Big ( \frac{-E}{ \mathbb{R} \mathcal{T}} \Big ),
\end{equation}
 where $c$ is the  concentration of the  reactant, $a$ is the  frequency factor  and $E$ is the  energy of activation of the reaction. Hence, incorporating \eqref{P3_E1} in \eqref{P3_E2} yields
 \begin{equation}\label{APS_eq1}
     \nabla^{2} \mathcal{T} = -\frac{ Q}{\9}\, c\,a  \exp\Big ( \frac{-E}{\mathbb{R} \mathcal{T}} \Big ).
 \end{equation}
After some approximations and  circular symmetry in \eqref{APS_eq1}, Chambre \cite{chambre1952solution} arrived at  the following Poisson-Boltzmann equation:
 \begin{equation*}\label{P3_E3}
     \frac{d^{2}\theta}{d z^{2}} +\frac{\kappa}{z} \frac{d\theta}{dz} = - \delta\,\exp (\theta),
 \end{equation*}
where \( \delta \) is the appropriate constant, and \( \kappa \) represents the geometry of the vessel:
 \begin{equation*}\label{P3_E4}
 \kappa=
\begin{cases}
    0, & \text{for infinite plane-parallel vessel}, \\
    1, & \text{for cylindrical vessel where length is much greater than radius}, \\
    2, & \text{for spherical vessel},
\end{cases}
 \end{equation*}
subject to the following boundary conditions:
\begin{equation*}\label{b.c.}
     \theta' (0) = 0,\;\;\; \theta(1) = 0.
\end{equation*}

To generalize the framework introduced in the above model, we consider the following class of nonlinear singular boundary value problems (NSBVPs) \cite{chambre1952solution}
\begin{equation}\label{P3_L1}
\1''(\z) + \frac{\alpha}{\z} \1' (\z) + g (\z, \1(\z)) = 0, \;\;\;\; 0 < \z  \leq 1, 
\end{equation}
subject to one of the specified boundary conditions:
\renewcommand{\labelenumi}{\alph{enumi})}
\begin{eqnarray}
&& \label{P3_eqn311} \1(0)  =  \7_1,~ \1(1) = \7_2, \\
&& \label{P3_eqn411} \1'(0)  =  \7_3, ~ m \1(1) +n \1'(1) = \7_4,~(m>0,n\geq0),
\end{eqnarray}
where $\alpha(\geq 0$) is the shape factor and $ \7_1,  \7_2, \7_3$, and $\7_4$ are arbitrary constants. The nonlinear term $g (\z, \1(\z))$ is continuous in $\z$ and sufficiently smooth with respect to $\1$. This equation also appears in diverse scientific and engineering fields, including the equilibrium of isothermal gas spheres \cite{chandrasekhar1957introduction}, thermal explosions in cylindrical vessels \cite{chambre1952solution}. Emden-Fowler type equations (EFTE), extensively studied by Fowler \cite{lane1870}, play a crucial role in modeling self-gravitating systems, stellar structures, and various other physical processes.

NSBVPs have attracted significant  attention since their inception, owing to their occurrence in a wide range of real-world applications. In recent years, numerous researchers have concentrated on developing efficient numerical methods to solve the NSBVPs \cite{Khuri2010, lepik2005, kayenat2022, baxley1998,chambre1952solution, chandrasekhar1957introduction, dickey1989, taylor2020,Majak2009,Ratas2021,Khan2021,Shahna2020,raza2020non}. The primary challenge in solving \eqref{P3_L1} stems from its singularity at $\z = 0$. Several numerical techniques have been developed to solve singular boundary value problems (SBVPs), including approaches based on cubic splines \cite{lSingh2014, Khuri2010}, which integrate modified decomposition methods with cubic spline collocation by dividing the domain into subintervals. Finite difference methods have also been successfully applied in this context \cite{kayenat2022}. Despite their effectiveness, these methods often involve complex formulations and substantial computational effort, particularly when addressing nonlinear problems. This underscores the need for alternative strategies that are both computationally efficient and easier to implement while preserving accuracy.

Alongside these developments, wavelet-based collocation methods have seen increasing application in the numerical solution of differential equations. These methods combine the localization properties of wavelets with the flexibility of collocation techniques, offering a powerful tool for numerical approximation. The wavelet based collocation approach, in particular, is valued for its simplicity, computational efficiency, and fast convergence, making it a preferred choice for a wide range of applications.


Over the years, numerous wavelets have been developed, including orthogonal and semi-orthogonal wavelets. Orthogonal wavelets encompass well-known families such as the Daubechies wavelet \cite{daubechies1992}, Haar wavelet  \cite{verma2019}, Legendre wavelet \cite{legendre2011} and Bernoulli wavelet \cite{bernoulli2014}. Although Daubechies wavelets are compactly supported and orthogonal, they lack a closed-form expression, making analytical differentiation and integration infeasible. In contrast, the Haar wavelet is the only real-valued wavelet that is both compactly supported and orthogonal with an explicit closed-form expression. 

The Haar wavelet, known for its simplicity and low computational cost, has been widely used in scientific and engineering applications. However, it suffers from a lack of smoothness due to its discontinuous nature. As a result, such bases do not provide significant performance improvements. To overcome these limitations associated with orthogonal wavelets, Goswami et al. \cite{goswami1995} applied the concept of semi-orthogonal wavelets in their work. In this framework, individual wavelets within the same subspace are not necessarily orthogonal to each other, but the wavelet subspaces remain mutually orthogonal. These semi-orthogonal wavelets possess both compact support and closed-form expressions. Notable constructions of such semi-orthogonal  wavelets can be found in \cite{chui1992,Jawerth1994,Sweldens1994,Unser1993} and the references there in. Examples of semi-orthogonal wavelets include the Taylor wavelet \cite{taylor2020} and the Fibonacci wavelet \cite{fibonacci2020}. 

In this paper, we introduce a novel semi-orthogonal wavelet, coined as semi-orthogonal Tribonacci wavelet (SOTW), and employ it to approximate functions belonging to the $L^2(\mathbb{R)}$ space, such as $\frac{\sin{\z}}{\z}$, $\z\log{\z}$, and others. Additionally, we propose a collocation method utilizing SOTW named semi-orthogonal tribonacci wavelet quasilinearization collocation method (SOTWQCM). Using SOTWQCM, we  solve Lane-Emden equations, third-order Emden-Fowler equations, and singular perturbation problems very effectively and accurately. The maximum error of the  SOTWQCM solution is compared with that of other approaches, including the non-standard finite difference (NSFD) method \cite{kayenat2022}, the bvp4c MATLAB solver \cite{shampine2003}, and wavelet-based methods such as the Taylor wavelet \cite{taylor2020}, Fibonacci wavelet \cite{fibonacci2020}, and Haar wavelet \cite{Verma2022}.  The comparison provides evidence highlighting the superiority of the proposed wavelet method.

The paper is structured as follows: Section \ref{Preliminary} focuses on introducing the tribonacci polynomial based on the tribonacci sequence and classification of wavelets. In section  \ref{TW}, we detail the construction of the SOTW and demonstrate its use in function approximation through several examples. In Section \ref{Method}, we provide the methodology and  convergence analysis, which include the quasilinearization technique along with the SOTW collocation method. Section \ref{Numerical} presents the numerical illustrations,  algorithm for SOTWQCM and test examples based on the SOTW. Section \ref{conclusion_P3} serves as the concluding part of the paper.
\section{Preliminary}\label{Preliminary}
This section explores the origins of the tribonacci sequence and tribonacci polynomials.
\subsection{The Tribonacci Polynomial}\label{tribonacci_poly}
The tribonacci numbers start with three initial terms, and each subsequent term is determined by the sum of the three preceding terms \cite{feinberg1963fibonacci}. In other words, the tribonacci sequence denoted by  $\f_{n}$, are generated by the following  recurrence relation:
\begin{equation*}\label{P3_E5}
    \f_{n} = \f_{n-1}+ \f_{n-2}+\f_{n-3}; \;\;\;  n \ge 4,
\end{equation*}
with the initial condition $\f_{1} = 1 = \f_{2}$, $\f_{3} = 2$.
 The tribonacci ratio $\frac{\f_{n+1}}{\f_{n}}$ converges to an irrational number $ 1.839286755214 \cdots$. 
 The tribonacci polynomials are expressed through the following general formula \cite{koshy2019fibonacci}:
 \begin{equation}\label{eq1_P3}
\e_{\8}(\z)=
    \begin{cases}
        1, & \text{if } \8 = 0 ,\\
        \z, & \text{if } \8 = 1, \\
        \z^{2}, & \text{if } \8 = 2, \\
        \z^{2}\e_{\8-1}(\z) + \z \e_{\8-2}(\z) + \e_{\8-3}(\z),  & \text{if } \8 \ge 3. \\
    \end{cases}
\end{equation}
\subsection{Orthogonal vs. Semi-orthogonal Wavelets}\label{tribonacci_wavelet}
In this section, we explore the fundamental concepts underlying wavelet theory, starting with basic definitions and properties essential for understanding their application in function approximation. We begin by revisiting orthogonal wavelets and we also introduce the notion of semi-orthogonal wavelets.
\subsubsection{Orthogonal Wavelets}\label{P3_OW}
A function $\hat{\psi}
 \in L^2(\mathbb{R})$ is  called  an orthogonal  wavelet if the family $ \hat{\psi}
_{j, k}(\z) = 2^{\frac{j}{2}} ~\hat{\psi}
 (2^j\z  - \2), \;\; j, \2 \in \mathbb{Z},$
   is an orthonormal basis of $L^2(\mathbb{R}),$ i.e., 
   \begin{equation*}\label{P3_W2}
       \langle \hat{\psi}
_{j,k}, \hat{\psi}
_{l,m} \rangle = \delta_{j,l}.\delta_{\2,m}, \;\; j,\2, l, m \in \mathbb{Z},
   \end{equation*}
   where $\delta_{j,l}$ represents the  Kronecker delta function. 
   
 Any arbitrary function  $ \4(\z) \in L^2(\mathbb{R})$ can be approximated by a wavelet series given by  \cite{chui1992}
   \begin{equation*}\label{P3_W1}
       \4(\z) = \sum_{j,\2 = - \infty}^{\infty}\h_{j,\2} \hat{\psi}
_{j, \2}(\z).
   \end{equation*}
   where
   \begin{equation*}
       \h_{j,\2} = \langle\4(\z), \hat{\psi}
_{j,\2} \rangle =\int_{-\infty}^{\infty} \4(\z) \overline{2^{\frac{j}{2}} ~\hat{\psi}
 (2^j \z  - \2)}\,d\z.
   \end{equation*}
   However, in general, we do not require  $\hat{\psi}
_{j,\2}(\z)$ to form an orthonormal basis of \( L^2(\mathbb{R}) \) to approximate a function \cite{chui1992}. In fact, a stable basis, as described below, is sufficient.
   \begin{definition} [\cite{chui1992}]\label{def1}
       A function $\psi \in L^2(\mathbb{R})$ is called  an R-function if $\w_{j,\2} (\z) = 2^{\frac{j}{2}} ~\w (2^j\z  - \2), \;\; j, \2 \in \mathbb{Z},$ is a Riesz basis of $ L^2(\mathbb{R})$, in the sense that  the linear span  of $\w_{j,\2}(\z), j, \2 \in \mathbb{Z}$ is dense in $L^2{(\mathbb{R})}$ and $\exists$  constant $C$ and $D$ exist with $0< C \leq D \leq \infty$, such that 
       \begin{equation*}
           C \|\h_{j,\2}\|_{l^2}^{2} \leq  \left\lVert \sum_{j = - \infty}^{\infty}\sum_{\2 = - \infty}^{\infty} \h_{j, \2} \w_{j, \2}  \right\rVert_{2}^{2}\leq D\|\h_{j,\2}\|_{l^2}^{2},
       \end{equation*}
        for all  doubly bi-infinte square  summable sequence $\h_{j,\2}$, i.e.,
        \begin{equation*}
            \|\h_{j,\2}\|_{l^2}^{2}:= \sum_{j = - \infty}^{\infty}\sum_{\2 = - \infty}^{\infty} |\h_{j,\2}|^2 < \infty.
        \end{equation*}
   \end{definition}
   Suppose that $\psi $ is an R-function then there is a unique Riesz basis $\left\{\w^{j,\2}  \right\}$ of $ L^2(\mathbb{R})$ which is a dual  to $\left\{\w_{j,\2}\right\}$ in the sense that \cite{chui1992}
   \begin{equation*}
        \langle \w_{j,k}, \w^{l,m} \rangle = \delta_{j,l}.\delta_{\2,m}, \;\; j,\2, l, m \in \mathbb{Z},
   \end{equation*}
  where $\delta_{j,l}$ represents the  Kronecker delta function. 
  
  Every function $\4(\z) \in L^2(\mathbb{R})$ has the following  series expansion as follows:
   \begin{equation}\label{series}
       \4(\z) = \sum_{j,\2 = - \infty}^{\infty} \langle\4(\z) , \w_{j,\2}(\z)\rangle\w^{j,\2}(\z).
   \end{equation}
   For \eqref{series} to qualify as  a wavelet series, there must exist a function $\overline{\psi} \in L^2(\mathbb{R})$ such that the dual basis $\left\{\w^{j,\2}\right\}$ appearing in series~\eqref{series} is generated from $\overline{\psi}$ by the relation $\w^{j,\2}(\z)=\overline{\w}_{j,\2}(\z) $, where
   \begin{equation}\label{dual}
       \overline{\w_{j, k}(\z)} := \overline{2^{\frac{j}{2}} ~\w (2^j\z  - \2)},
   \end{equation}
   as described in \cite{chui1992}. However, since this condition is not  necessarily satisfied, we conclude that  the series \eqref{series} does not in general, represent a wavelet series.
    So, we obtain the  series expansion as follows: 
    \begin{equation}\label{wavelet series}
       \4(\z) = \sum_{j,\2 = - \infty}^{\infty} \langle\4(\z) , \w_{j,\2}(\z)\rangle \overline{\w}_{j,\2}(\z).
   \end{equation}
   \begin{definition} [\cite{chui1992}]\label{wavelet}
       An R-function  $\psi \in L^2(\mathbb{R})$ is called a wavelet if  there exist a function $\overline{\psi} \in L^2(\mathbb{R})$ such that $\left\{\w_{j,\2}\right\}$ and $\left\{\overline{\w_{j,\2}}\right\}$ are  the dual basis  of $L^2{(\mathbb{R})}$. If $\psi$ is a wavelet then $\overline{\psi}$ is called  a dual wavelet  related  to $\psi$.
   \end{definition}
   Every wavelet $\psi$, orthogonal or not, generate  a wavelet series  representation of $\4(\z) \in L^2(\mathbb{R})$,
   \begin{equation}\label{P3_L2}
       \4(\z) = \sum_{j,\2 = - \infty}^{\infty} \h_{j,\2}\w_{j,\2}(\z),
   \end{equation}
   where  the wavelet coefficient $\h_{j,\2}$ are given as follows:
   \begin{equation}\label{P3_L3}
       \h_{j,\2} = \langle\4(\z), \w_{j,\2} \rangle.
   \end{equation}
   \begin{definition}[\cite{chui1992}]\label{semi-orthogonal wavelet }
       A wavelet $\psi$  in $L^2 (\mathbb{R})$ is called a semi-orthogonal wavelet  if the Riesz basis $\left\{\w_{j,\2}\right\}$ it generates satisfies 
       \begin{equation*}\label{P3_W4}
           \langle \w_{j,\2}, \w_{l,m} \rangle = 0, \;\;\; j \neq l, \; j,\2, l, m \in \mathbb{Z}.
       \end{equation*}
   \end{definition}
   \section{ The Semi-orthogonal Tribonacci Wavelet} \label{TW}
The mother wavelet generates a family of functions through scaling  and shifting of itself, and these are referred to as wavelets. The discrete wavelet  family is presented in the following manner:  
 \begin{equation*}\label{P3_E7}
     \w _{\2, \3} (\z) =  |a_{0}| ^{\frac{\2}{2}} ~ \w ( a_{0} ^{\2} ~\z - \3 ~b_{0}),
 \end{equation*}
 in which $ \w _{\2, \3} (\z)$  serves as a wavelet basis in  $ L^{2} (\mathbb{R})$.
 
The SOTW are defined in the following manner:
 \begin{equation}\label{eq2_P3}
\w_{\3, m}(\z)=
    \begin{cases}
         2 ^{\frac{\2 - 1}{ 2}} \hat{\e}_{\g} (2^{\2 -1 } \z  - \3 +1 ) , & \text{if }  \frac{\3 -1 }{ 2^{\2 -1 }}  \leq \z < \frac{\3}{ 2^{\2 -1}} ,\\
        0, & \text{otherwise}, \\
    \end{cases}
\end{equation}
with
\begin{equation*}\label{P3_E8}
    \hat{\e}_{\g} ( \z ) =  \frac{1}{\sqrt{\Z_{\g}}} \e _{\g}(\z),
\end{equation*}
and 
\begin{equation*}\label{P3_E9}
    \Z _{\g} = \int_{0}^{1} \e _{\g}^{2} (\z) \,d\z.
\end{equation*}
Here, $ \frac{1}{\sqrt{\Z_{\g}}}$ is a  normalization factor, $\g = 0, 1, 2, \cdots, M - 1$ represents the order of the tribonacci polynomial $T_{\g} (\z)$,  $\2 = 1, 2, 3, \cdots$, and $\3 = 1, 2, 3, \cdots, 2^{\2 - 1}$. In this context, $\2$ signifies the resolution level, respectively. 

For instance, let us take $ \2 = 3 , M = 3 $, we  obtain the SOTW basis as follows:
\begin{equation*}\label{P3_E11}
\begin{array}{ll}
\displaystyle
\w_{1, 0}(\z) =
\begin{cases}
2, & \text{if } 0 \leq \z < \frac{1}{4}, \\
0, & \text{otherwise},
\end{cases}
&
\displaystyle
\w_{1, 1}(\z) =
\begin{cases}
8\sqrt{3}, & \text{if } 0 \leq \z < \frac{1}{4}, \\
0, & \text{otherwise},
\end{cases}
\\[3ex]
\displaystyle
\w_{1, 2}(\z) =
\begin{cases}
32 \sqrt{5} \z^2, & \text{if } 0 \leq \z < \frac{1}{4}, \\
0, & \text{otherwise},
\end{cases}
&
\displaystyle
\w_{2, 0}(\z) =
\begin{cases}
2, & \text{if } \frac{1}{4} \leq \z < \frac{1}{2}, \\
0, & \text{otherwise},
\end{cases}
\\[3ex]
\displaystyle
\w_{2, 1}(\z) =
\begin{cases}
2\sqrt{3}(4\z - 1), & \text{if } \frac{1}{4} \leq \z < \frac{1}{2}, \\
0, & \text{otherwise},
\end{cases}
&
\displaystyle
\w_{2, 2}(\z) =
\begin{cases}
2\sqrt{5}(4\z - 1)^2, & \text{if } \frac{1}{4} \leq \z < \frac{1}{2}, \\
0, & \text{otherwise},
\end{cases}
\\[3ex]
\displaystyle
\w_{3, 0}(\z) =
\begin{cases}
2, & \text{if } \frac{1}{2} \leq \z < \frac{3}{4}, \\
0, & \text{otherwise},
\end{cases}
&
\displaystyle
\w_{3, 1}(\z) =
\begin{cases}
2\sqrt{3}(4\z - 2), & \text{if } \frac{1}{2} \leq \z < \frac{3}{4}, \\
0, & \text{otherwise},
\end{cases}
\\[3ex]
\displaystyle
\w_{3, 2}(\z) =
\begin{cases}
2\sqrt{5}(4\z - 1)^2, & \text{if } \frac{1}{2} \leq \z < \frac{3}{4}, \\
0, & \text{otherwise},
\end{cases}
&
\displaystyle
\w_{4, 0}(\z) =
\begin{cases}
2, & \text{if } \frac{3}{4} \leq \z < 1, \\
0, & \text{otherwise},
\end{cases}
\\[3ex]
\displaystyle
\w_{4, 1}(\z) =
\begin{cases}
2\sqrt{3}(4\z - 3), & \text{if } \frac{3}{4} \leq \z < 1, \\
0, & \text{otherwise},
\end{cases}
&
\displaystyle
\w_{4, 2}(\z) =
\begin{cases}
2\sqrt{5}(4\z - 3)^2, & \text{if } \frac{3}{4} \leq \z < 1, \\
0, & \text{otherwise}.
\end{cases}
\end{array}
\end{equation*}
\begin{figure}[H] 
  \centering
  \includegraphics[width=1.0\textwidth]{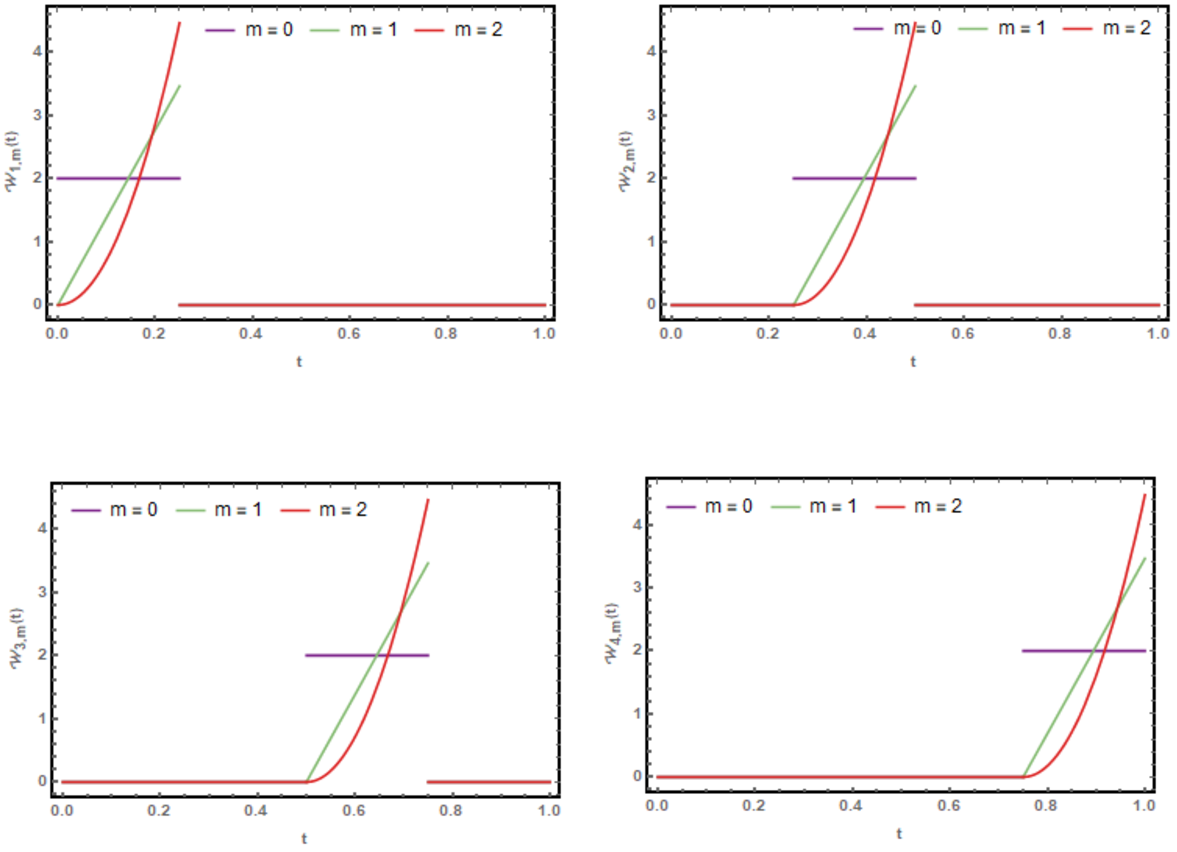} 
 \caption{Graph of the SOTW for $\2 = 3$,  $M= 3$.} 
 \label{P3_fig1}
\end{figure}
\subsection{Function Approximation using SOTW}\label{FAUTW}
An arbitrary function  $\4(\z)  \in  L^{2} [0, 1] $ can be approximately expanded in the following way:
   \begin{eqnarray}
&& \label{eq4_P3}
      \4 (\z) \approx \sum_{\3 = 1}^{2^{\2 - 1}} \sum_{\g = 0}^{M - 1} \h_{\3, \g} \w _{\3, \g} (\z),\\
&& \label{P31_eq4}
       \4(\z)  \approx \sum_{\3 = 1}^{2^{\2 - 1}} [ \h_{\3, 0} \w_{\3, 0}+\h_{\3, 1} \w_{\3, 1}+ \cdots+ \h_{\3, M -1 } \w_{\3, M - 1}],\\
       && \label{P31_eq5}
       \4(\z)  \approx  \h_{1, 0} \w_{1, 0}+ \cdots+ \h_{1, M -1 } \w_{1, M - 1}+\cdots+ \h_{2^{\2-1}, 0} \w_{2^{\2 - 1}, 0}+ \cdots+\h_{2^{\2 - 1}, M -1 } \w_{2^{\2 - 1}, M - 1}.
       \end{eqnarray}
        To simplify the expression, we multiply the above equation \eqref{P31_eq5} by $\w_{\3, \g}(\z)$ and integrate between $0$ to $1$, where $\3 = 1, 2, 3, 4, \cdots, 2^{\2 - 1},$ $ \;m = 0, 1, 2, 3, \cdots, M - 1,$ then we get  $ 1 \times 2^{\2 -1} M$ vectors $G$ as follows:
      \begin{eqnarray}
    && G = \left[ \4_{\3, m} \right], \;\; \3 = 1, 2, 3, 4, \cdots, 2^{\2 - 1}, \;\;m = 0, 1, 2, 3, \cdots, M - 1,\\
    && \4_{\3, m} = \int_{0}^{1} \4(\z) \w_{\3, m}(\z) \,d\z.
  \end{eqnarray}
  Now, we define the vectors  $B$ and $\w (\z)$, each of size $ (1 \times 2^{\2 -1} M)$  which are given as follows:
\begin{eqnarray}
&& \label{eq5_P3}
B= [ \h _{1, 0}, \h_{1, 1}, \cdots , \h_{1, M -1}, \h_{2,0}, \h_{2,1}, \cdots, \h_{2, M-1}, \cdots, \h_{2^{\2-1},0}, \h_{2^{\2-1},1}, \h_{2^{\2-1},2}, \cdots, \h_{2^{\2-1},M-1}], \\
   &&  \w(\z) = [ \w _{1, 0}, \w_{1, 1}, \cdots , \w_{1, M -1}, \w_{2,0}, \cdots, \w_{2, M-1}, \cdots, \w_{2^{\2-1},0}, \w_{2^{\2-1},1}, \w_{2^{\2-1},2}, \cdots, \w_{2^{\2-1},M-1}].
\end{eqnarray}
In equation, \eqref{eq4_P3}, \eqref{P31_eq4}, \eqref{P31_eq5} and  \eqref{eq5_P3}, $\h_{ \3, m}$, $\3 = 1, 2, 3, \cdots $ 
 $2^{\2 - 1}$, $m = 0,1, 2, 3, 4, \cdots, M -1$ are the SOTW  expansion coefficient of the function $\4( \z )$ in the $ n^{th} $ subinterval  $ [\frac{ \3 - 1}{ 2^{\2 - 1}} , \frac{n}{2^{\2 - 1}} ]$. Using  \eqref{eq5_P3}, the vector coefficients  $ \h_{n, m}$ can be achieved as :
 \begin{align*}
     B = G D^{-1},
 \end{align*}
 where,
 \begin{eqnarray}
  && D = \langle \w, \w \rangle =  \int_{0}^{1} \w^{T}(\z)\w(\z)  \,d\z.
  \end{eqnarray}
\begin{example}\label{P3_Ex1}
Consider a function $g(t)$ given by
\begin{align*}
      \4(\z ) =  \frac{\sin \z}{\z}.
\end{align*}
\end{example}
Upon performing the computations utilizing the SOTW, the resulting SOTW series with SOTW coefficients for $ \2 = 1$ and $ M = 5$ are as follows: 
\begin{eqnarray*}\label{SOTW1}
     \4(\z)  \approx  0.99101 \, \w_{1, 0} (\z)+  0.000262147 \, \w_{1, 1} (\z) - 0.0824239 \, \w_{1, 2} (\z)+  0.0184607 \, \w_{1, 3} (\z) - 0.000668257 \, \w_{1, 4} (\z).
\end{eqnarray*}
In Figure \ref{P3_fig2}, we present a comparison between the original function and its approximation using the SOTW  for $ \2 = 1$ and $ M = 5$. The figure also illustrates the absolute error, highlighting the accuracy of the approximation.

\begin{example}\label{P3_Ex2}
Consider a function $g(t)$ given by
\begin{align*}\label{P3_E12}
      \4(\z ) =  \z \log\z.
\end{align*}
\end{example}
Upon performing the computations utilizing the SOTW, the resulting SOTW  series with SOTW  coefficients for $\2 = 1 $ and $ M= 7$ are as follows:
\begin{eqnarray*}\label{SOTW2}
     \4(\z)  \approx  14.999 \, \w_{1, 0} (\z) - 17.4469 \, \w_{1, 1} (\z) +11.9286 \, \w_{1, 2} (\z) - 34.2645 \, \w_{1, 3} (\z) +51.9891 \, \w_{1, 4} (\z) \\\quad \quad\quad - 38.6191 \, \w_{1, 5} (\z)+12.6573 \, \w_{1, 6} (\z).
\end{eqnarray*}
In Figure \ref{P3_fig2}, we present a comparison between the original function and its approximation using the SOTW for $\2 = 1 $ and $ M= 7$. The figure also illustrates the absolute error, highlighting the accuracy of the approximation.
\begin{figure}[H] 
  \centering
  \includegraphics[width=1.0\textwidth]{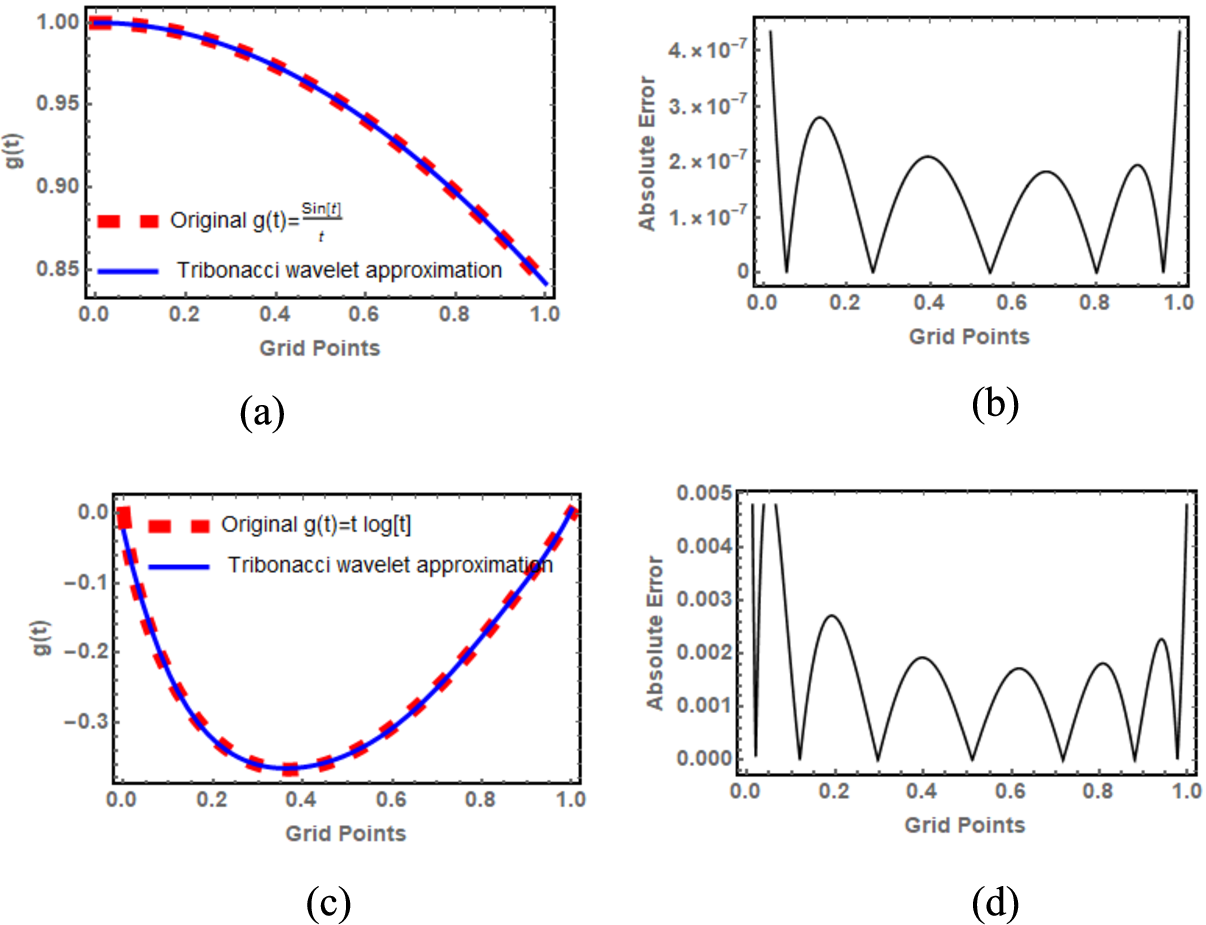} 
  \caption{Plot of function approximation of $\4(\z)$ and corressponding absoulte error  using SOTW method.} 
  \label{P3_fig2}
\end{figure}
\section{Methodology }\label{Method}
This section presents the methodology for computing solution of  nonlinear singular  differential equation of the different form using the semi-orthogonal tribonacci wavelet collocation method (SOTWCM) combined with quasilinearization. Quasilinearization is first used to linearize the NSBVPs. The numerical solutions are then computed using SOTW, followed by the discritization via the collocation method. Additionally, we present the convergence analysis of the SOTW.
\subsection{Quasilinearization}\label{Quasilinrization}
The quasilinearization approach converts a nonlinear differential equations (DE) into a series of linear DE, with the solution determined as the limit of this series. For a comprehensive discussion on quasilinearization, one may refer to the work of \cite{verma2019} and \cite{mandelzweig2001}. Quasilinearization allows for the conversion of any nonlinear ordinary differential equation of the form 
\begin{equation}\label{P3_diffeqn1}
\1''(\z) + \frac{\alpha}{\z} \1' (\z) + g (\z, \1(\z)) = 0, \;\;\;\; 0 < \z  \leq 1, 
\end{equation}
subject to one of the specified boundary conditions:
\renewcommand{\labelenumi}{\alph{enumi})}
\begin{eqnarray}
&& \label{P3_eqn3} \1(0)  =  \7_1,~ \1(1) = \7_2, \\
&& \label{P3_eqn4} \1'(0)  =  \7_3, ~ m \1(1) +n \1'(1) = \7_4,
\end{eqnarray}
into a sequence of linear DE. The \((r + 1)\)th iterative approximation \(\1_{\R+1}(\z)\) to the solution of \eqref{P3_diffeqn1}, subject to boundary conditions given by \eqref{P3_eqn3}, or \eqref{P3_eqn4}, is obtained by solving the associated linear DE. This process is formalized in the following theorem.
\begin{theorem}\label{theorem1}
Consider a function \( g(\z, \1(\z)) \) that is continuous in \( \z \) and twice continuously differentiable with respect to \( \1(\z) \). Let \( \1(\z) \) be the solution to the differential equation \eqref{P3_diffeqn1}, under the specified  boundary conditions \eqref{P3_eqn3}, or \eqref{P3_eqn4}. By using the quasilinearization technique, we describe the iterative scheme for solving the differential equation  \eqref{P3_diffeqn1} with the given boundary conditions through the following set of equations:
\begin{eqnarray}
&& \label{P3_eqn5} \1''_{\R +1}(\z) + \frac{\alpha}{\z} \1'_{\R+1}(\z) = - g (\z, \1_{\R}(\z)) + (\1_{\R+1} - \1_{\R}) \left( - g_{\1}(\z, \1_{\R}(\z)) \right),~ \R = 0,1,2,\cdots, \\
&& \label{P3_eqn7} \1_{\R +1}(0) = \7_1, \quad \1_{\R +1}(1) = \7_2, \\
&& \label{P3_eqn8} \1'_{\R +1}(0) = \7_3, \quad m \1_{\R +1}(1) + n \1_{\R +1}'(1) = \7_4,
\end{eqnarray}
where $\alpha(\geq 0)$, $ \7_1, \7_2, \7_3, \7_4, m, $ and $n$ are arbitrary constants and $g_{\1}$ represents the partial derivative of the function $g$ with respect to $\1$. 
\end{theorem}

\begin{proof}\label{P3_E13}
The proof follows directly from the Taylor series expansion. 
\end{proof}
\subsection{The SOTWQCM}\label{SOTWQCM}
As a result of the aforementioned theorem, the DE \eqref{P3_diffeqn1} is obtained in its linearized form. We then apply the SOTW  collocation method to solve the linear differential equation \eqref{P3_eqn5}, subject to the one of the boundary conditions \eqref{P3_eqn7}, or \eqref{P3_eqn8}. The second order derivative present in the DE is approximated by the SOTW  basis as follows:
\begin{align}\label{P3_eqn10}
    \1 '' _{\R+1 }(\z)  \approx  \sum_{n = 1}^{2^{\2-1}}\sum_{m = 0}^{M-1} \h_{n,m}\w_{n,m}(\z).
\end{align}
Integrating the  above equation  \eqref{P3_eqn10} two times from $0$ to $\z$, we get the  solution of the differential equation,
 \begin{eqnarray}
\label{P3_eqn11} && \1 ' _{\R+1 }(\z) \approx  \1 ' _{\R+1 }(0)+ \sum_{n = 1}^{2^{\2-1}}\sum_{m = 0}^{M-1} \h_{n,m}\int_{0}^{\z} \w_{n,m}(\z)  \, d\z
,\\
\label{P3_eqn12} &&  \1  _{\R+1 }(\z) \approx   \1  _{\R+1 }(0)+ \z \1 ' _{\R+1 }(0)+ \sum_{n = 1}^{2^{\2-1}}\sum_{m = 0}^{M-1} \h_{n,m}\int_{0}^{\z} \int_{0}^{\z}\w_{n,m}(\z) \,d\z \, d\z,\\
\label{P3_equn1} &&  \1  _{\R+1 }(\z) \approx   \1  _{\R+1 }(0)+ \z \1 ' _{\R+1 }(0)+ \sum_{n = 1}^{2^{\2-1}}\sum_{m = 0}^{M-1} \h_{n,m} \, P^{2}\w _{\3, \g} (\z),\\
\label{P3_equn2} &&  \1  _{\R+1 }(\z) \approx    C_{T}(\z)+ \sum_{n = 1}^{2^{\2-1}}\sum_{m = 0}^{M-1} \h_{n,m} \, P^{2}\w _{\3, \g} (\z),
\end{eqnarray}
where $ C_{T}(\z)$ stands for the constants of integration, $ P^{2}\w _{\3, \g} (\z)$ represents the second order integration of the $\w _{\3, \g}(\z)$.

Based on the specific nature of the problem, one of the boundary conditions \eqref{P3_eqn7}, or \eqref{P3_eqn8} is used to determine the unknowns. The equations are discretized  by replacing $ \z$ with the collocation points defined as follows:
\begin{eqnarray} \label{collocation}
    \z_{l}= \frac{2l -1}{2^{\2}M}, \;\;\; l = 1, 2, 3, \cdots, (2^{\2-1}M),
\end{eqnarray}
 and substituting them into linearized  equation \eqref{P3_eqn5}. This leads to the following linear system of $ (2^{\2-1}M)$ algebraic equation with $ (2^{\2-1}M)$ unknowns:
 \begin{equation}\label{linear_system}
\begin{aligned}
&\x_{1} (\h _{1, 0}, \h_{1, 1}, \h_{1,2}, \cdots , \h_{1, M -1}, \h_{2,0}, \h_{2,1}, \cdots, \h_{2, M-1}, \cdots, \h_{2^{\2-1},0}, \h_{2^{\2-1},1}, \h_{2^{\2-1},2}, \cdots, \h_{2^{\2-1},M-1}) = 0, \\
&\x_{2} (\h _{1, 0}, \h_{1, 1}, \h_{1,2}, \cdots , \h_{1, M -1}, \h_{2,0}, \h_{2,1}, \cdots, \h_{2, M-1}, \cdots, \h_{2^{\2-1},0}, \h_{2^{\2-1},1}, \h_{2^{\2-1},2}, \cdots, \h_{2^{\2-1},M-1}) = 0, \\
&\vdots ~~~~~~~~~~~~~~~~~~~~\vdots ~~~~~~~~~~~~~~~~~~~~\vdots~~~~~~~~~~~~~~~~~~~~\vdots\\
&\x_{2^{\2-1}M} (\h _{1, 0}, \h_{1, 1}, \h_{1,2}, \cdots , \h_{1, M -1}, \h_{2,0}, \h_{2,1}, \cdots, \h_{2, M-1}, \cdots, \h_{2^{\2-1},0}, \h_{2^{\2-1},1}, \h_{2^{\2-1},2}, \cdots, \h_{2^{\2-1},M-1}) = 0. 
\end{aligned}
\end{equation}
By solving the above system of linear equations, we obtain the SOTW coefficients. These are then substituted back into the equation \eqref{P3_eqn12} to compute the numerical solution. The process is repeated  for $ \R = 1, 2, 3, \cdots$ until the desired accuracy is achieved.
\subsection{Convergence Analysis}\label{CGT}
 We present the following theorem to establish the convergence of the SOTWQCM for the given NSBVPs \eqref{P3_L1}. 
Using SOTW collocation method, we have
 \begin{equation}\label{cgt1_P3}
        h(\z) = \1'' (\z)  =  \sum_{\3 = 1}^{\infty} \sum_{\g = 0}^{\infty} \h_{\3, \g} \w _{\3, \g} (\z).
    \end{equation}
  Integrating the above equation \eqref{cgt1_P3} two times, we get  
   \begin{equation}\label{cgt2_P3}
        \1 (\z)  =  \sum_{\3 = 1}^{\infty} \sum_{\g = 0}^{\infty} \h_{\3, \g} P^{2}\w _{\3, \g} (\z) + C_{T}(\z),
    \end{equation}
    where $C_{T}(\z)$ stands for the constants of integration, $ P^{2}\w _{\3, \g} (\z)$ represents the second order integration of the $\w _{\3, \g}(\z)$.
\begin{theorem}\label{P3_theorem}
    If a continuous function  $ h (\z) = \frac{d^2 \1 }{d\z^2}$  be a square integrable  function defined on the interval $[0,~ 1]$, such that  it is bounded  by $\beta$, i.e., $|h(\z)| \leq  \beta$, for every $\z \in [0,~ 1]. $ Then, the SOTW expansion of $ h(\z)$  in \eqref{eq4_P3}  converges.
\end{theorem}
\begin{proof}\label{P3_proof}
    In \eqref{cgt2_P3}, truncating the expansion, we have
     \begin{equation}\label{cgt3_P3}
        \1_{\2,\5}(\z)  =  \sum_{\3 = 1}^{2^{\2-1}} \sum_{\g = 0}^{M-1} \h_{\3, \g} P^{2}\w _{\3, \g} (\z) + C_{T}(\z).
    \end{equation}
    Now, we define the error  as follows:
    \begin{equation}\label{P3_cgt1}
\left|\6_{\2, \5}\right| \; =\left| \1(\z) - \1_{\2,\5}(\z)\right| =\;  \left| \sum_{\3 = 2^{\2}}^{\infty} \sum_{\g = \5}^{\infty} \h_{\3, \g} P^{2}\w _{\3, \g} (\z) \right|.
\end{equation}
The $L^{2}$-norm of the given  error function is defined as,
\begin{equation}\label{P3_cgt2}
\lVert \6_{\2, \5}\rVert_{2}^{2} \; =  \int_{0}^{1} \Big (\sum_{\3 =2^{\2} }^{\infty}  \sum_{\g = \5}^{\infty}   \h_{\3, \g} P^{2}\w _{\3, \g} (\z)\Big)^{2}\, d\z.
\end{equation}
After expanding above equation, we have
\begin{equation}\label{P3_cgt3}
    \lVert \6_{\2, \5}\rVert_{2}^{2} \; =   \sum_{\3 = 2^{\2}}^{\infty}  \sum_{\g = \5}^{\infty} \sum_{\1 = 2^{\2}}^{\infty}  \sum_{\R = \5}^{\infty} \h_{\3, \g}  \h_{\1, \R} \int_{0}^{1}  P^{2}\w _{\3, \g}(\z) ~ P^{2}\w _{\1, \R}(\z) \;  d\z.
\end{equation}
Now, as $\z \in [0,~1],$ we have
\begin{eqnarray}
  && \label{P3_eq3}
     \left|  P^{2}\w _{\3, \g}(\z)  \right| \leq  \int_{0}^{\z}  \int_{0}^{\z} \w _{\3, \g}(\z)   d\z d\z,\\
       && \label{P3_eq4}
       \left|  P^{2}\w _{\3, \g}(\z)  \right| \leq  \int_{0}^{\z}  \int_{0}^{1}  \w _{\3, \g}(\z)   d\z d\z.
     \end{eqnarray}
Now, by using equation  \eqref{eq2_P3}, we have
\begin{equation*}\label{P3_E14}
     \left|  P^{2}\w _{\3, \g}(\z)  \right| \leq  2^{\frac{\2 - 1}{2}}  \frac{1}{\sqrt{z_{\g}}}\int_{0}^{\z}  \int_{\frac{\3-1}{2^{\2-1}}}^{\frac{\3}{2^{\2-1}}} T_{\g}(2^{\2-1} \z - \3 +1) d\z d\z.
\end{equation*}
By using, change of variable, $2^{\2 -1} \z - \3 +1 = y$, and taking the  relation $ T_{\g}(y)=  \frac{T_{\g+1}'(y)}{\g+1}$, we get,
\begin{equation*}\label{P3_E114}
     \left|  P^{2}\w _{\3, \g}(\z)  \right| \leq  2^{\frac{1-\2}{2}}  \frac{1}{\sqrt{z_{\g}}}\int_{0}^{\z}\int_{0}^{1}  \frac{T_{\g+1}'(y)}{\g+1} \;dy\; d\z.
\end{equation*}
 Now, using continuity and boundedness  of $ T_{\g}(y)$ implies that there exist a real constant $\theta$ such that $ \int_{0}^{1}T'_{\g+1}(y) \  dy = \theta$, we get the bound as follows:
\begin{equation*}\label{P3_E15}
     \left|  P^{2}\w _{\3, \g}(\z)  \right| \leq 2^{\frac{1-\2}{2}} \frac{\theta}{\sqrt{z_{\g}} (\g+1)}.
\end{equation*}
    The SOTW expansion coefficients of $h(\z)$ can be written  as follows:
    \begin{eqnarray*}\label{P3_E17}
     \begin{aligned}
         & \h_{\3, \g} = \int_{\frac{n-1}{2^{\2-1}}}^{\frac{n}{2^{\2-1}}} h(\z) 2 ^{\frac{\2 - 1}{ 2}} \frac{{\e}_{\g} (2^{\2 -1 } \z  - \3 +1 )}{\sqrt{z_{\g}}}  \ d\z,\\
         & \h_{\3, \g} = \frac{2 ^{\frac{1-\2}{2}}}{\sqrt{z_{\g}}} \int_{0}^{1} h \Big(\frac{y+n-1}{2^{\2-1}} \Big) T_{\g}(y)   \ dy,\\
          & | \h_{\3, \g}| \leq \frac{2 ^{\frac{-\2}{2}}}{\sqrt{z_{\g}}}  \sqrt{2} \beta  \int_{0}^{1} \left|\frac{T'_{\g+1}(y)}{\g+1}   \right | \ dy,\\
           &| \h_{\3, \g}| \leq \frac{2 ^{\frac{-\2}{2}}}{\sqrt{z_{\g}}}  \sqrt{2} \beta  \frac{\theta}{\g+1}.
    \end{aligned}
    \end{eqnarray*}
Now, putting these values in equation \eqref{P3_cgt3}, we get
\begin{equation*}\label{P3_E20}
     \lVert \6_{\2, \5}\rVert_{2}^{2} \; \leq 2^{-2 \2} \beta^{2}\theta^{4} \sum_{\3 = 2^{\2}}^{\infty} 
 \left(\sqrt{\frac{2}{z_{\3}}} \right)^{2} \sum_{\g = \5}^{\infty} \frac{1}{(\g+1)^{2}} \sum_{\1 = 2^{\2}}^{\infty}   \left(\frac{\sqrt{2}}{\sqrt{z_{\1}}} \right)^{2}\sum_{\R = \5}^{\infty}  \frac{1}{(\R+1)^{2}},
\end{equation*}
since all four series converge, and as $ \2, \5 \rightarrow \infty$, then $\lVert \6_{\2, \5}\rVert  \rightarrow 0$.
    \end{proof}
\section{Numerical Illustration }\label{Numerical}
In this section, the proposed method is corroborated using several test examples. All experiments are performed in Mathematica $11.3$ on a system equipped with a $64-$bit Intel Core $i7$ CPU and $8$ GB of RAM. We present the numerical results with \(\2 = 1\) and different $M$ across all test cases. We present the detailed steps to compute the numerical solutions in algorithmic form. The maximum absolute error is also presented to demonstrate the accuracy and features of the proposed method, defined as follows:
 \begin{definition}\label{definition}
For a function \(\1(\z)\) defined over the domain \([0, 1]\), we define two types of errors  using the SOTWQCM solution \(\1_{M}(\z)\):

\begin{enumerate}
    \item \textbf{Maximum Absolute Error:}
    \begin{equation*}
    \delta_{M}^{max} = \max_{\z_i \in [0, 1]} \delta_{M}(\z_i),
    \end{equation*}
    where $\delta_{M}(\z_i) = \left| \1(\z_i) - \1_{M}(\z_i) \right|$ is the absolute error and \(\1(\z)\) is the exact solution.

    \item \textbf{Maximum Residual Error:}
    \begin{equation*}
    e^{max} = \max_{\z_i \in [0, 1]} e_{res}(\z_i),
    \end{equation*}
    where $e_{res}(\z_i) = \left|  {\1_{M}}''(\z_i) + \frac{\alpha}{\z_i} {\1_{M}}' (\z_i) + g ((\z_i), {\1_{M}}(\z_i))\right|$ is the residual error and \(\1_{M}(\z)\) is the SOTWQCM solution.
\end{enumerate}\end{definition}
\subsection{Algorithm for SOTWQCM}\label{Algorithm}
 This subsection provides the primary steps of the suggested method  in detail for solving   NSBVPs \eqref{P3_diffeqn1}. The essential steps are outlined below.
 \begin{description}
      \item[Step 1:] Using the quasilinearization technique, the nonlinear SBVP \eqref{P3_diffeqn1}, along with the boundary conditions \eqref{P3_eqn3}, and \eqref{P3_eqn4}, is  converted into a sequence of linear differential equations \eqref{P3_eqn5}.
    \item [Step 2:] Implement the SOTWQCM approach on the linearized  differential equations, incorporating the appropriate boundary conditions specific to the problem.
\item [Step 3:] The equations \eqref{P3_eqn10}, \eqref{P3_eqn11} and  \eqref{P3_eqn12} are discretized by replacing \(\z\) with the collocation points \eqref{collocation} and then inserting these equations into \eqref{P3_eqn5}.
\item [Step 4:]  We arrive at a system of linear algebraic equations of size $ (2^{\2-1}M)\times (2^{\2-1}M)$. 
\item [Step 5 :] Solve these equations, starting with an initial guess $\1_0$, to determine the SOTW coefficients.
\item [Step 6 :]  Substitute the obtained wavelet coefficients into solution equation  to derive the first approximate numerical solution.
\item [Step 7 :] Repeat steps $2-6$ for $\R$ iterations to refine the SOTW  solution until the desired accuracy is achieved.
  \end{description}

\setcounter{example}{0} 
\renewcommand{\theexample}{\ref{Numerical}.\color{red}\arabic{example}}
\begin{example}\label{EX1_P3}
Let us assume the following  NSBVPs:
\begin{equation*}\label{P1_eqn38}
    \1''(\z) + \frac{2}{\z} \1' (\z) + \1^{5} (\z) = 0, ~\1'(0) = 0,~\1(1) = \sqrt{\frac{3}{4}}.
\end{equation*}
\end{example}
The exact solution of this  NSBVP is $\1 (\z) = \sqrt{\frac{3}{3+\z^{2}}}$. This problem plays a significant role in the analysis of stellar structures \cite{chandrasekhar1957introduction}. Table \ref{table2:example1_P3} provides detailed results, showcasing the CPU time (in seconds) and maximum absolute error for \( \2 = 1 \) across different values of \( M \) using the initial guess vector $\left[1, 1, \dots, 1\right]$. In Table \ref{table3:example1_P3}, the error comparison between  SOTWQCM, the NSFD method  \cite{kayenat2022}, Taylor wavelet \cite{taylor2020}, and the bvp4c MATLAB solver \cite{shampine2003} is provided, highlighting the superior performance of our approach.
\begin{table}[H]
  \centering
  \caption{Performance Analysis: The CPU time (s) and $\delta_{M}^{max}$ for the SOTWQCM are presented for \( \2 = 1 \) and various \( M \) values for the given Example \ref{EX1_P3}.}\label{table2:example1_P3}
  \resizebox{16.5cm}{0.669cm}
  {
    \begin{tabular}{|c|c|c|c|c|c|}
    \hline
     $ M$ & $ 2$ & $4$ & $ 6$ & $ 8 $ & $  9$ \\ \hline
    CPU time (s) &  $ 0.5 $&  $ 0.827$& $ 1.641$ & $ 5.126$& $ 16.671$ \\ \hline
    $\delta_{M}^{max}$   & $ 0.00177828$ & $ 0.0000285214$ & $0.0000267767$ & $ 2.17433E - 07$& $ 6.90402E - 08$\\ \hline
\end{tabular}}
\end{table}
\begin{table}[H]
  \centering
  \caption{ Comparison of the $\delta_{M}^{max}$ achieved by SOTWQCM with  other methods for the given  Example \ref{EX1_P3}.}\label{table3:example1_P3}
  \resizebox{17.5cm}{0.8 cm}{
    \begin{tabular}{|c|c|c|c|c|c|c|}
    \hline
     \multicolumn{2}{|c|}{  SOTWQCM}  &  \multicolumn{2}{|c|}{  NSFD \cite{kayenat2022} }     & \multicolumn{2}{|c|}{ bvp4c \cite{shampine2003} }  & Taylor wavelet \cite{taylor2020}\\
  \cline{1-7}   $ \2 =1,~ M = 9$ & $ \2 =1,~ M = 10$& $ N = 256$ & $ N = 512$ & $ N = 256$  &  $ N = 512$  &  $M = 9$ \\
    \hline 
    $6.90402E - 09$ & $6.74738E - 10$ & $ 1.27473E - 06$ & $3.18683E - 07$ & $ 5.00081E - 06 $ & $ 5.02068E - 06$ &  $ 6.52E - 06$
    \\ \hline
 \end{tabular}}
\end{table}

\begin{example}\label{EX2_P3}
Let us assume  the  following NSBVPs \cite{akvreview2020}:
\begin{equation*}\label{P3_eqn13}
    \1''(\z) +\frac{1}{\z} \1'(\z) + e^{\1(\z)} = 0, ~ \1'(0) = 0, ~\1(1) = 0.
\end{equation*}
\end{example}
 The exact solution of this  NSBVP  is $\1 (\z) = 2 \ln \Big( \frac{4 - 2 \sqrt{2}}{( 3 - 2 \sqrt{2}) \z ^{2}+1}\Big )$. This problem is relevant to the study of thermal explosions occurring in cylindrical vessels \cite{chambre1952solution}. Starting with an initial guess vector of $[0, 0, \cdots, 0]$, Table \ref{table2:example2_P3} provides detailed results, showcasing the CPU time (in seconds) and maximum absolute error for \(\2 = 1\) and various values of \(M\). 
 
 Table \ref{table3:example2_P3} presents a comparison of the errors obtained using  SOTWQCM, NSFD \cite{kayenat2022}, Fibonacci wavelet \cite{fibonacci2020}, and the bvp4c MATLAB solver \cite{shampine2003}. The results clearly demonstrate the superiority of our method.
\begin{table}[H]
  \centering
  \caption{Performance Analysis: The CPU time (s) and $\delta_{M}^{max}$ for the SOTWQCM are presented for \( \2 = 1 \) and various \( M \) values for the given Example \ref{EX2_P3}.}\label{table2:example2_P3}
  \resizebox{16.5cm}{0.678cm}
  {
    \begin{tabular}{|c|c|c|c|c|c|}
    \hline
     $M$ & $2$ & $4$ & $6$ & $8$ & $9$ \\ \hline
    CPU time (s)&  $0.719$&  $ 6.031$& $ 3.64$ & $ 4.312$& $ 5.641$ \\ \hline
    $\delta_{M}^{max}$   & $ 0.00141529$ & $ 4.63284E - 06$ & $ 5.68708E - 06$ & $ 2.0193E - 08$& $ 5.88003E - 09$\\ \hline
\end{tabular}}
\end{table}
\begin{table}[H]
  \centering
  \caption{ Comparison of the $\delta_{M}^{max}$ achieved by SOTWQCM with  other methods for the given Example \ref{EX2_P3}.}\label{table3:example2_P3}
  \resizebox{17.5cm}{0.83 cm}{
    \begin{tabular}{|c|c|c|c|c|c|c|}
    \hline
     \multicolumn{2}{|c|}{ SOTWQCM } &  \multicolumn{2}{|c|}{ NSFD \cite{kayenat2022} }     & \multicolumn{2}{|c|}{ bvp4c \cite{shampine2003} } & Fibonacci   \cite{fibonacci2020}\\
  \cline{1-7}   $ \2 =1,~ M = 9$ & $ \2 =1,~ M = 10$& $ N = 256$ & $ N = 512$ & $ N = 256$  &  $ N = 512$ & $M = 5$ \\
    \hline 
    $5.88003E - 09$ & $ 7.79022E - 11$ & $ 8.75833E - 06$ & $ 2.65175E - 06$ & $ 1.431237922E - 05 $ & $ 1.33574449E - 05$ & $ 3.38689E - 07$ 
    \\ \hline
 \end{tabular}}
\end{table}
\begin{example}\label{EX3_P3}
Let us consider the  following NSBVPs:
\begin{equation*}\label{P3_eqn14}
    \1''(\z) +\frac{3}{\z} \1'(\z) + \Big ( \frac{1}{8 \1^{2}} -\frac{1}{2} \Big ) = 0, ~ \1'(0) = 0, ~\1(1) = 1.
\end{equation*}
\end{example}
The exact solution to this  NSBVP is not available. NSBVPs of this type have been explored in work such as \cite{dickey1989} and \cite{baxley1998}. This equation plays a crucial role in the analysis of displacements and stresses within a shallow membrane cap. Using the initial guess vector \( [0.5, 0.5, 0.5, \dots, 0.5] \), the SOTWQCM maximum residual error for \( M = 4, 5 \) along with the comparison with HWCM \cite{verma2019} and bvp4c MATLAB solver \cite{shampine2003} are presented in Table \ref{table6:example3_P3}.
\begin{table}[H]
  \centering
  \caption{ Comparison of the $e^{max}$  achieved by SOTWQCM with  other methods for the given Example \ref{EX3_P3}.}\label{table6:example3_P3}
  \resizebox{16.5cm}{0.72cm}{
    \begin{tabular}{|c|c|c|c|c|c|}
    \hline
     \multicolumn{2}{|c|}{ SOTWQCM } &  \multicolumn{2}{|c|}{ HWCM \cite{verma2019} }     & \multicolumn{2}{|c|}{ bvp4c \cite{shampine2003} } \\
  \cline{1-6}   $ \2 =1,~ M = 4$ & $ \2 =1,~ M =5 $& $ J = 2$ & $ J = 4$ & $ n = 100$  &  $ n = 200$  \\
    \hline 
    $3.90847E - 07$ & $4.01346E - 08$ & $ 0.000410692$ & $0.0000861912$ & $ 7.539199E - 01 $ & $ 7.539199E - 01$
    \\ \hline \end{tabular}}
\end{table}

\begin{example}\label{EX4_P3}
    Let us assume  the following  first kind of nonlinear EFTE of third order \cite{wazwaz2015}
    \begin{equation*}
        \1'''(\z) +\frac{6}{\z} \1''(\z) +\frac{6}{\z^{2}} \1'(\z) - 6 (10 +2 \z^{3}+ \z^{6})e^{- 3 \1(\z)} = 0, \;\;  \1(0) = 0,\;\; \1'(0) = \1''(0) =  0.
    \end{equation*}
    \end{example}
    The exact solution of this NSBVP is $ \1 (\z) = \log(1+{\z^{3}})$. Using an initial guess vector $[0, 0, 0, \cdots, 0]$, the solution of the third-order nonlinear EFTE is obtained by SOTWQCM. Table \ref{table41:example2_P3} presents detailed results, including CPU time (in seconds) and the maximum absolute error for \( \2 = 1 \) across different values of \( M \).
    
    In Table \ref{table2:example4_P3}, we compare the maximum absolute error for \( \2 = 1 \) across various values of \( M \) using SOTWQCM, the bvp4c MATLAB solver \cite{shampine2003}, and the Haar wavelet collocation method \cite{singh2020}. The results indicate that as \( M \) increases, the error decreases, demonstrating the effectiveness of larger \( M \) values in improving accuracy.
\begin{table}[H]
  \centering
  \caption{Performance Analysis: The CPU time (s) and $\delta_{M}^{max}$ for the SOTWQCM are presented for \( \2 = 1 \) and various \( M \) values for the given Example \ref{EX4_P3}.}\label{table41:example2_P3}
  \resizebox{16.5cm}{0.68cm}
  {
    \begin{tabular}{|c|c|c|c|c|c|}
    \hline
     $M$  & $4$ & $5$& $7$ & $9$ & $11$ \\ \hline
    CPU time (s)&  $ 8.36$& $ 18.796$  & $60.406$ & $ 119.891$& $ 200.765$ \\ \hline
    $\delta_{M}^{max}$   & $ 0.000377157$ & $0.000139582$ & $ 5.81953E - 06$ & $ 3.64928E - 07$& $ 8.17307E - 08$\\ \hline
\end{tabular}}
\end{table}
\begin{table}[H]
  \centering
  \caption{Comparison of the $\delta_{M}^{max}$ achieved by SOTWQCM with  other methods for the given Example \ref{EX4_P3}.}\label{table2:example4_P3}
  \resizebox{17.5cm}{0.75 cm}{
    \begin{tabular}{|c|c|c|c|c|c|c|}
    \hline
     \multicolumn{3}{|c|}{ SOTWQCM} & \multicolumn{2}{|c|}{ HWCM \cite{singh2020}}   & \multicolumn{2}{|c|} { bvp4c  \cite{shampine2003}} \\
  \cline{1-7}  $\2 =1, M = 10$ & $ \2 =1,  M = 12$ & $  \2 =1, M = 13$   & $J = 2$ & $J = 9$ & $n = 100$& $n = 200$ \\ 
   \cline{1-7}  $1.05861E-07$ & $ 1.14506E-08$ & $ 3.89139E-09 $   & $ 1.627E-03$ & $ 1.098E-07$  & $1.880297E-08$ & $1.879164E-08$\\ \hline
\end{tabular}}
\end{table}

\begin{example}\label{EX5_P3}
    Let us assume  the following  second kind of nonlinear EFTE of third order \cite{wazwaz2015}
    \begin{equation*}
        \1'''(\z) - \frac{2}{\z} \1''(\z)  = \1^{2}(\z) + \1(\z) - \z^{6} e^{2 \z} + 7 \z^{2} e^{\z} + 6 \z e^{\z}  - 6 e^{\z}, \;\;  0 \leq \z \leq 1,  
    \end{equation*}
    under the specified boundary conditions,
    \begin{equation*}
        \1(0) = 0,\;\;\1'(0) =  0, \;\; \1(1) = e.
    \end{equation*}
    \end{example}
    The exact solution of this nonlinear SBVP is $ \1 (\z) = \z^{3} e ^{\z}$. Using an initial guess vector \([0,0,0, \dots, 0]\), the solution of the third-order nonlinear EFTE is obtained using SOTWQCM. Table \ref{table411:example2_P3} presents detailed results, including CPU time (in seconds) and the maximum absolute error for \( \2 = 1 \) across different values of \( M \).  

In Table \ref{table2:example5_P3}, the maximum absolute error is compared for \( \2 = 1 \) and various values of \( M \) using SOTWQCM, the Haar wavelet collocation method \cite{singh2020}, and the bvp4c MATLAB solver \cite{shampine2003}. The results show that increasing \( M \) leads to a reduction in error, highlighting the improved accuracy with larger values of \( M \).
\begin{table}[H]
  \centering
  \caption{Performance Analysis: The CPU time (s) and $\delta_{M}^{max}$ for the SOTWQCM are presented for \( \2 = 1 \) and various \( M \) values for the given Example \ref{EX5_P3}.}\label{table411:example2_P3}
  \resizebox{16.5cm}{0.68cm}
  {
    \begin{tabular}{|c|c|c|c|c|c|}
    \hline
     $M$  & $4$ & $5$& $7$ & $9$ & $11$ \\ \hline
    CPU time (s)&  $ 12.907$& $ 31.437$  & $96.094$ & $ 226.671$& $ 433.609$ \\ \hline
    $\delta_{M}^{max}$   & $ 0.000123202$ &$0.000046141$  & $ 1.11937E - 06$ & $ 4.23564E - 06$& $ 6.19893E - 09$\\ \hline
\end{tabular}}
\end{table}
\begin{table}[H]
  \centering
  \caption{Comparison of $ \delta_{M}^{max}$ using  SOTWQCM with other methods for the given Example \ref{EX5_P3}.}
  \resizebox{17.5cm}{0.8cm}{
  \begin{tabular}{|c|c|c|c|c|c|c|c|}
    \hline
  \multicolumn{4}{|c|}{ \text{  SOTWQCM}} & \multicolumn{2}{|c|}{\text{HWCM} \cite{singh2020}} & \multicolumn{2}{|c|}{\text{ bvp4c} \cite{shampine2003}}     \\
  \cline{1-8}   $ \2 =1,M =5$ &  $ \2 =1,M = 10$ & $ \2 =1, M = 12$ & $\2 =1, M = 13$   & $J = 3$ & $J = 9$ & $n = 100$ & $n = 200$  \\ 
   \cline{1-8} $ 0.000046141$ & $ 3.6534E-08$ & $ 4.19022E-09$ & $ 6.6657E-10 $ &  $ 1.41522E-03$ & $ 3.473E-07$ &  $ 2.417528E-07$ & $ 2.490911E-07$  \\ \hline
\end{tabular}}
\label{table2:example5_P3}
\end{table}

\begin{example}\label{EX8_P3}
Consider the following   linear singular  perturbation problem as follows:
\begin{equation*}
      -\epsilon   \1''(\z) +\1(\z) = \z,\;\;\; 0 \leq \z \leq 1,
\end{equation*}
 under the specified boundary conditions,
\begin{align*}
    \1 ( 0) = 1,\;\;\; \1(1) = 1+ e^{\frac{-1}{\sqrt{\epsilon}}}.
\end{align*}
\end{example}
The exact solution of this linear singular perturbation BVP is $ \1 (\z) = \z + e^{\frac{- \z}{\sqrt{\epsilon}}}$. This problem has been addressed using the  SOTWQCM, for \(\2 = 1\) with several   values of \(M\) and \(\epsilon\), where \(\epsilon = \frac{1}{32}\), \(\epsilon = \frac{1}{64}\), and \(\epsilon = \frac{1}{128}\). The CPU time (in seconds), maximum  absolute error for  various values of \(\epsilon\) and \(M\) are presented in Table~\ref{table:example8_P3}. By increasing the value of \(M\), better accuracy and efficiency can be achieved. 

In Table \ref{table:example8_P3}, we also presents the comparison of maximum absolute error using SOTWQCM  and bvp4c MATLAB solver \cite{shampine2003} by choosing the different value of $\epsilon$.
\begin{table}[H]
  \centering
  \caption{Comparison of   $\delta_{M}^{max}$, CPU time (s) for $\2 = 1$ across different value of  $M$ and \(\epsilon\)\ using SOTWQCM with bvp4c MATLAB solver \cite{shampine2003} of Example \ref{EX8_P3}.}
  \label{table:example8_P3}
 {\setlength{\tabcolsep}{10pt} 
    \begin{tabular}{|c|c|c|c|c|}
     \cline{1-5}     $\epsilon$ &  $\2= 1$ & $\delta_{M}^{max}$ & CPU time (s)   &  bvp4c \cite{shampine2003}\\ 
     \hline
     & $M = 10$ & $	4.52493E-07	$ & $ 29.187$  & $2.5816207E-04$ \\
      \cline{2-5}  
      $\frac{1}{32}$ & $M = 12$ & $	5.63276E-08	$ & $ 51.767$ &$1.24420491E-04$ \\
      \cline{2-5}
       & $M = 14$ & $	5.75644E-09	$ & $85.235$ & $6.636629E-05$\\
      \hline
        & $M = 10$ & $3.80832E-06 $ & $29.671$ &  $7.0150918E-05$\\
     \cline{2-5}  
      $\frac{1}{64}$& $M = 12$ & $	3.21547E-07	$ & $51.937$ & $1.5422064E-04$\\
      \cline{2-5}  
      & $M = 14$ & $	2.79873E-08$ & $ 82.483$  & $2.321384E-04$\\
      \hline
       & $M = 10$ & $0.0000504751	$ & $29.0$ & $2.436347E-04$ \\
     \cline{2-5}  
       $\frac{1}{128}$ & $M = 12$ & $ 2.2555E-06	$ & $ 50.859$  & $1.2441761E-04$\\
      \cline{2-5}  
      & $M = 14$ & $	3.00767E-07	$ & $ 82.874$ & $6.6364803E-05$\\
      \hline 
\end{tabular}}
\end{table}

\section{Conclusion} \label{conclusion_P3}
In this work, we introduce a new class of semi-orthogonal  wavelets and  refer them as Tribonacci wavelet.  An efficient SOTWQCM is developed and applied to solve non-linear singular differential equations, including Lane-Emden equations, Emden-Fowler equations, and singularly perturbed problems. The solutions obtained using this method, along with error analysis and comparison with recently developed techniques, indicate that the SOTWQCM  yields the highest accuracy in approximate solutions, even for small values of  \(M\) and \(\2\). This highlights its ability to provide reliable and precise results with reduced computational effort. Numerical results have been included to illustrate the precision and reliability of the proposed approach. This underscores the superiority of the proposed SOTWQCM method compared to other techniques such as the bvp4c MATLAB solver \cite{shampine2003}, various wavelet-based methods \cite{taylor2020,fibonacci2020,singh2020}, NSFD \cite{kayenat2022},  while offering reduced computational effort and time. Therefore, the SOTWQCM method offers a straightforward and effective strategy to solve non-linear singular differential equations.
\section*{Acknowledgement}
The first author is very grateful to the DST-Inspire fellowship (DST-Inspire Code - IF210618) for providing the financial support necessary to conduct this research. The first author is also thankful to Narendra Kumar, a research scholar at IIT Jodhpur for his support and guidance. 

\section*{Conflict of interest}
The authors declare that they have no conflicts of interest to report.

\bibliographystyle{plain}
\bibliography{P3ankita.bib}

\end{document}